\documentclass[12pt,a4paper]{amsart}

\usepackage[centering,body={6in,9in}]{geometry}

\usepackage{graphicx}

\theoremstyle{plain}
  \newtheorem{thm}{Theorem}[section]
  \newtheorem{cor}[thm]{Corollary}
  \newtheorem{lem}[thm]{Lemma}
  
\theoremstyle{definition}

\theoremstyle{remark}

\DeclareMathOperator{\bdim}{dim_B}

\DeclareMathOperator{\dist}{dist}
\DeclareMathOperator{\hd}{d_H}
\newcommand{\Num}{\mathfrak{N}}

\title{On the Separation Structure of Lalley-Gatzouras Fractals}

\begin{document}
	
\author{Li-Feng Xi}
\address{Department of Mathematics, Ningbo University, Ningbo,
	Zhejiang, 315211, P.~R.\ China}
\email{xilifengningbo@yahoo.com}

\author{Jun Jie Miao}
\address{Department of Mathematics, Shanghai Key Laboratory of PMMP, East China Normal University, No. 500, Dongchuan Road, Shanghai 200241, P. R. China}
\email{jjmiao@math.ecnu.edu.cn}

\author{Ying Xiong}
\address{Department of Mathematics, South China University of Technology,
	Guangzhou, 510641, P.~R. China}
\email{xiongyng@gmail.com}



\thanks{Ying Xiong is the corresponding author.}

\thanks{Supported by National Natural Science Foundation of China (Grant Nos
	11201152, 11371329, 11471124), NSF of Zhejiang Province (No. LR13A010001), K. C. Wong Magna Fund in Ningbo University, Science and Technology Commission of Shanghai Municipality (No. 13dz2260400) and Morningside Center of Mathematics.}

\begin{abstract}
	We obtain a necessary and sufficient condition for Lalley-Gatzouras sets to be uniform disconnected. This enable us to find all Lalley-Gatzouras sets which are quasisymmetrically equivalent to the Cantor ternary set. As another application, we also study the limit behavior of the gap sequence of Lalley-Gatzouras sets.
\end{abstract}

\maketitle

\section{Introduction}\label{sec:intro}

Although every compact totally disconnected perfect set is homeomorphic to the Cantor ternary set, such sets in a metric space may have quite different geometric structures. So a natural question is to seek a general way to characterize the sets which are roughly like the Cantor ternary set. David and Semmes~\cite{DavSe97} answer this question in the sense of \emph{quasisymmetric equivalence}. Two metric spaces  are said to be quasisymmetrically equivalent if there is a quasisymmetric mapping from one onto the other. A mapping~$f$ from $(X,d)$ to $(Y,\rho)$ is called quasisymmetric if it is an injection and there exists a homeomorphism $\eta\colon[0,\infty)\to[0,\infty)$ such that for any $x,y,z\in X$ and $t>0$,
\[ d(x,y)\le td(x,z)\implies\rho(f(x),f(y))\le\eta(t)\rho(f(x),f(z)). \]
We refer to~\cite{DavSe97,Heino01} and the references therein for detailed description of quasisymmetric equivalence and quasisymmetric mappings.

With the notion of uniform disconnectedness, David and Semmes~\cite{DavSe97} proved that every compact, doubling, uniformly disconnected and uniformly perfect set is quasisymmetrically equivalent to the Cantor ternary set (see~\cite[Proposition~15.11]{DavSe97}). Note that every subset of Euclidean spaces is doubling, while Xie, Yin and Sun~\cite{XiYiS03} prove that every self-affine set is uniformly perfect if it is not singleton. It is therefore of interest to look at what self-affine sets are uniformly disconnected.

As an important object in fractal geometry and related fields, self-affine sets are composed of smaller affine copies of themselves. In the self-affine construction, contraction ratios may assume different values in different directions. This fact causes huge difficulties to obtain results which hold for all self-affine sets (see~~\cite{PerSo00}). Consequently, there are two approaches in the study of self-affine sets: finding results for \emph{generic} sets (almost all in some sense, see Falconer~\cite{Falco88,Falco92}) or for \emph{specific} sets (see~\cite{Baran07,Bedfo84,FenWa05,LalGa92,McMul84,Solom98}). 

In this paper, we study a special kind of self-affine sets on~$\mathbb R^2$: Lalley-Gatzouras sets, which were first studied by Lalley and Gatzouras~\cite{LalGa92}, as a generalization of McMullen sets~\cite{Bedfo84,McMul84}. We obtain a necessary and sufficient condition for such sets to be uniformly disconnected. This enable us to find all Lalley-Gatzouras sets which are quasisymmetrically equivalent to the Cantor ternary set. 

As another application, we use this result to study the gap sequence of Lalley-Gatzouras sets. Gap sequence first appear in the work of Besicovitch~\cite{BesTa54} on cutting-out sets in the line. Rao, Ruan and Yang~\cite{RaRuY08} generalized this concept to sets in higher dimensional spaces. In the follow-up study, the gap sequence is widely used to explore properties of fractals. In particular, it is often related to box dimension, especially for self-similar sets or cutting-out fractals, see~\cite{BesTa54,DeWaX15,Falco97,GaMoS07,MiXiX17,RaRuY08,XioWu09}. Roughly speaking, if a set~$E$ has a regular separation structure (e.g., $E$ is a self-similar sets satisfying the strong separation condition), then its gap sequence~$\{\alpha_k\}_k$ and its box dimension~$\bdim E$ are related by~$\alpha_k\asymp k^{-1/\bdim E}$, i.e., $c^{-1}\alpha_k\le k^{-1/\bdim E}\le c\alpha_k$ for a constant~$c$. We will prove this result for Lalley-Gatzouras sets by making use of uniform disconnectedness. 

This paper is organized as follows. We recall some basic definitions in Section~\ref{ssec:defn}; and present main results in Section~\ref{ssec:rslt}. In Section~\ref{sec:UD}, we study the uniform disconnectedness and prove Theorem~\ref{t:UD}. Finally, we discuss gap sequence and prove Theorem~\ref{t:GS} in Section~\ref{sec:GS}.

\subsection{Basic definitions}\label{ssec:defn}

We recall the definitions of Lalley-Gatzouras sets, uniform disconnectedness and gap sequence in this subsection.

\subsubsection*{Lalley-Gatzouras sets \textup{(see~\cite{LalGa92})}}

Given integers $m\ge2$ and $n_i\ge0$ with $i=1,2,\ldots,m$. Let $b_i>0$ be real numbers for $i=1,2,\ldots,m$ such that $\sum_{i=1}^m b_i=1$. Write $d_1=0$ and $d_i=b_1+b_2+\dots+b_{i-1}$ for $i=2,3,\dots,m$. Let
\[\mathcal{D}=\Bigl\{(i,j)\colon 1\le i\le m\ \text{for $n_i>0$ and $1\le j\le n_i$}\Bigr\}.\]
Let $\{a_{ij}\}_{(i,j)\in\mathcal D}$ and $\{c_{ij}\}_{(i,j)\in\mathcal D}$ be two sequences of nonnegative numbers such that $1>b_i> a_{ij}>0$ for each $(i,j)\in\mathcal D$, $\sum_{j=1}^{n_i}a_{ij}\le 1$ for each $i$, $c_{i(j+1)}-c_{ij}\ge a_{ij}$ for each $(i,j)\in\mathcal D$ and $1-c_{in_i}\ge a_{in_i}$ for each $i$.

For each $(i,j)\in \mathcal{D}$, we define a self-affine transformation on~$\mathbb R^2$ by
\begin{equation}\label{eq:Sk}
S_{ij}(x,y)^T=\begin{pmatrix}
a_{ij} & 0 \\
0 & b_i 
\end{pmatrix}\binom{x}{y}+ \binom{c_{ij}}{d_i},
\end{equation}
then the family $\{S_{ij}\}_{(i,j)\in \mathcal{D}}$ forms a self-affine iterated function system. According to Hutchinson~\cite{Hutch81}, there exists an attractor $E$, called a \emph{Lalley-Gatzouras set} \cite{LalGa92}, such that
\begin{equation}
E=\bigcup_{(i,j)\in \mathcal{D}}S_{ij}(E).
\end{equation}

Note that the hypotheses on $\{a_{ij}\}$, $\{b_i\}$ and $\{c_{ij}\}$ guarantee that the open rectangles $\{S_{ij}\bigl((0,1)^2\bigr)\}_{(i,j)\in\mathcal D}$ are pairwise disjoint subsets of $(0,1)^2$ with edges parallel to the $x$- and $y$-axes, are arranged in rows of height $b_i$, and have height greater than width. See Figure~\ref{fg1} present the first level structure of a Lalley-Gatzouras set.

\begin{figure}[h]
	\centering
	\includegraphics[width=10cm]{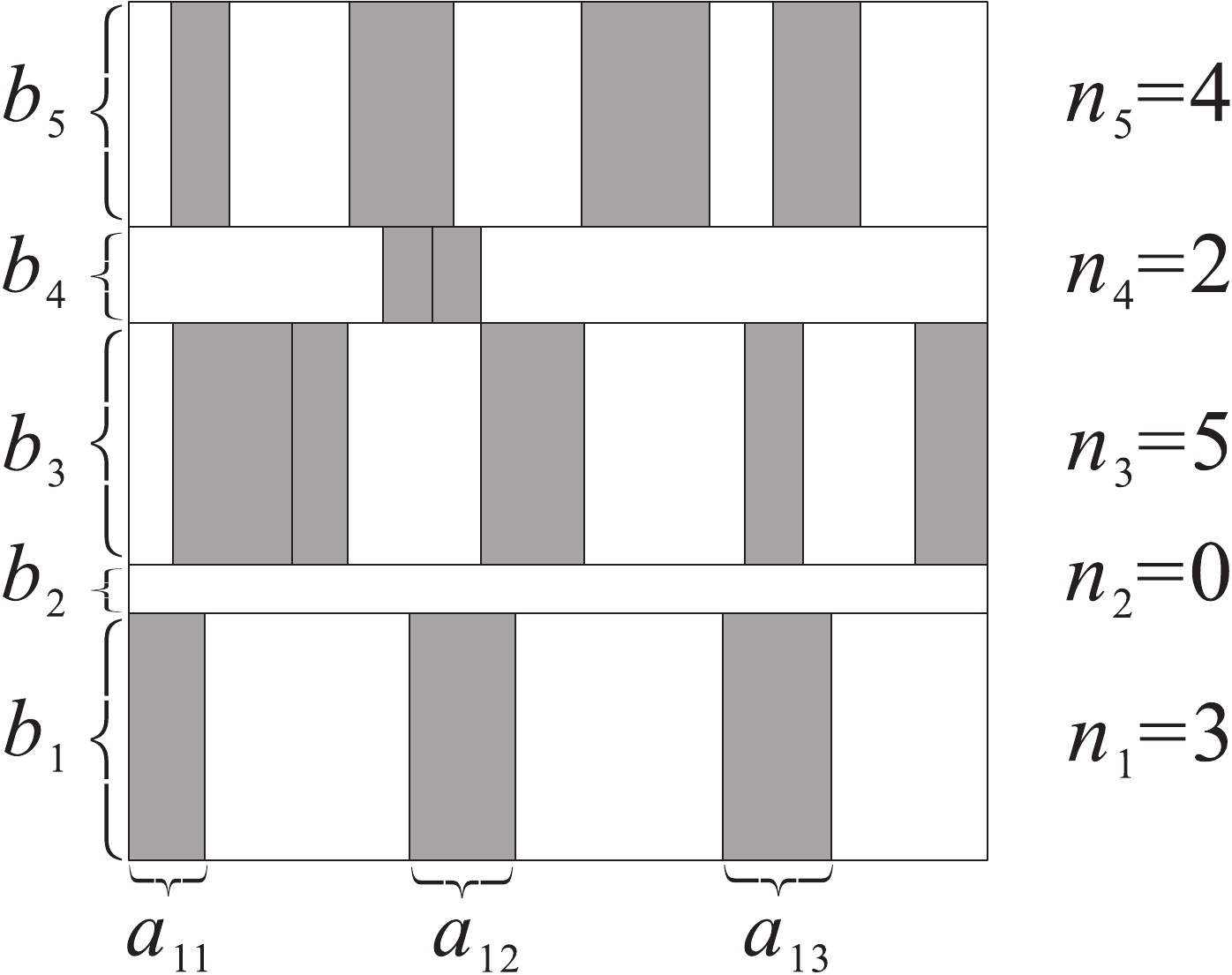}
	\caption{The first level structure of a Lalley-Gatzouras set}\label{fg1}
\end{figure}

As stated in \cite{LalGa92}, let $s_1\in\mathbb{R}$ be the unique real such that
\begin{equation}\label{def_s1}
\sum_{i=1, n_i\ne 0}^m b^{s_1}_i=1.
\end{equation}
Then the box dimension of $E$ is $\bdim E=s$, where $s$ is given by
\begin{equation}\label{eq:bdim}
\sum_{i=1, n_i\ne 0}^m \sum_{j=1}^{n_i}b^{s_1}_ia_{ij}^{s-s_1}=1.
\end{equation}

\subsubsection*{Uniform disconnectedness \textup{(see~\cite{Heino01,DavSe97})}}

Let $E$ be a subset of a metric space $(X,\rho)$. $E$ is called uniformly disconnected if there is a constant $C>1$ so that, for each $\xi\in E$ and $r>0$, we can find a subset~$A$ of~$E$ such that $B(\xi,r/C)\cap E\subset A\subset B(\xi,r)$ and $\dist(A,E\setminus A)\ge r/C$.

\subsubsection*{The gap sequence}

The gap sequence of cutting-out sets in the line was first studied by Besicovitch~\cite{BesTa54}. Rao, Ruan and Yang~\cite{RaRuY08} generalized this concept to sets in higher dimensional spaces.

Let $A$ be a compact subset of $\mathbb{R}^2$. For distinct $x,y\in A$, we say $x$ and~$y$ are \emph{$\delta$-equivalent} if there exists a sequence of points $a_0=x, a_1,\dots, a_k=y$ of $A$ such that $|a_{i+1}-a_i|\le\delta$ for $i=0,1,\dots,k-1$. Let $\Num(\delta)$ be the cardinality of the set of $\delta$-equivalent classes of $A$. Clearly the mapping $\Num\colon\mathbb{R}\to \mathbb{N}$ is non-increasing. We write $\Num(\delta^-)=\lim_{h\to 0^+}\Num(\delta-h)$. We say a sequence $\{\alpha_k\}_{k\ge1}$ is the \emph{gap sequence} of $A$ if the elements of $\{\alpha_k\}_{k\ge1}$ are made of the jump points of the function $\mathfrak{\delta}$ as $\delta$ decreases with multiplicity $\Num(\delta^-)-\Num(\delta)$,  that is, the value of $\delta$ is in the gap sequence $\{\alpha_k\}_{k}$ if and only if $\Num(\delta)<\Num(\delta^-)$ and the multiplicity
\[
\operatorname{card}\{k: \alpha_k=\delta\}=\Num(\delta^-)-\Num(\delta).
\]
We also write $\Num(\delta, A)$ to emphasize the dependence on the set $A$.

For an example, the gap sequence of the Cantor ternary set is 
\[ 1/3,1/9,1/9,\dots,\underbrace{1/3^n,\dots,1/3^n}_{2^{n-1}},\dots. \]

\subsection{Main results}\label{ssec:rslt}

Our first result concerns the uniform disconnectedness of Lalley-Gatzouras sets.
\begin{thm}\label{t:UD}
	A Lalley-Gatzouras set $E$ is uniformly disconnected if and only if $E$ is totally disconnected and there exists an $i\in\{1,2,\dots,m\}$ such that $n_i=0$.
\end{thm}

David and Semmes~\cite[Proposition~15.11]{DavSe97} proved that every compact, doubling, uniformly disconnected and uniformly perfect set is quasisymmetrically equivalent to the Cantor ternary set. Since self-affine sets are always doubling and uniformly perfect (see ~\cite{XiYiS03}), Theorem~\ref{t:UD} implies
\begin{cor}\label{c:QE}
	A Lalley-Gatzouras set $E$ is quasisymmetrically equivalent to the Cantor ternary set if and only if $E$ is totally disconnected and there exists an $i\in\{1,2,\dots,m\}$ such that $n_i=0$.
\end{cor}
Geometrically, the condition $n_i=0$ for some $i\in\{1,2,\dots,m\}$ means that there are some horizontal line segments in~$[0,1]^2$ do not intersect with the Lalley-Gatzouras set. Corollary~\ref{c:QE} implies that the Lalley-Gatzouras sets that satisfy this condition are essentially different from that do not satisfy this condition.

As another application of Theorem~\ref{t:UD}, we obtain the relationship between the gap sequence of Lalley-Gatzouras set and its box dimension. 
\begin{thm}\label{t:GS}
	Let $E$ be a totally disconnected Lalley-Gatzouras set with $n_i=0$ for some $i\in\{1,2,\dots,m\}$ and $\{\alpha_k\}_k$ the gap sequence of~$E$. Then
	\[ \alpha_k\asymp k^{-1/\bdim E}. \]
\end{thm}
Here the symbol $\alpha_k\asymp\beta_k$ means that there exists a constant $c>1$ such that $c^{-1}\beta_k\le\alpha_k\le c\beta_k$ for all $k$.

\section{Uniform disconnectedness of Lalley-Gatzouras sets}\label{sec:UD}

We shall prove Theorem~\ref{t:UD} in this section. Let $E$ be a Lally-Gatzouras set with the IFS $\{S_{ij}\}_{(i,j)\in\mathcal D}$ defined by~\eqref{eq:Sk}. 

\subsection{Proof of necessity} The following equivalent condition for uniform disconnectedness is useful (see \cite[\S14.24]{Heino01}). $E$ is uniformly disconnected if and only if there is $\epsilon_0>0$ such that no pair of distinct points $\xi,\xi'\in E$ can be connected by an $\epsilon_0$-chain in~$E$, where an $\epsilon_0$-chain connecting~$\xi$ and~$\xi'$ in~$E$ is a sequence of points $\xi=\xi_0$, $\xi_1$, $\ldots$, $\xi_n=\xi'$ in~$E$ satisfying $|\xi_i-\xi_{i+1}|\le\epsilon_0|\xi-\xi'|$ for all $i=0,1,\dots,n-1$.

We prove the necessary part of Theorem~\ref{t:UD} by using the above equivalent condition. Suppose that $n_i\ne0$ for all $i=1,2,\dots,m$. It suffices to show that, for any $\epsilon_0\in(0,1)$, there are two points $\xi,\xi'$ can be connected by an $\epsilon_0$-chain in~$E$.

For this, observe that the condition $n_i\ne0$ for all $i$ implies that the projection of~$E$ to $y$-axis is~$[0,1]$. Hence we can find a sequence $\{(x_k,k/n)\}_{0\le k\le n}\subset E$, where $n$ is a positive integer with $1/n<\epsilon_0/2$. Pick an affine mapping~$T=S_{i_1j_1}\circ S_{i_2j_2}\circ\dots\circ S_{i_\ell j_\ell}$, where $S_{i_kj_k}$ defined by~\eqref{eq:Sk}, such that 
\[ \frac{a_{i_1j_1}a_{i_2j_2}\cdots a_{i_\ell j_\ell}}{b_{i_1}b_{i_2}\cdots b_{i_\ell}}\le\frac{\epsilon_0}{2}. \]

We claim that $\xi=T(x_0,0)$ and $\xi'=T(x_n,1)$ are two points connected by the $\epsilon_0$-chain $\{T(x_k,k/n)\}_{0\le k\le n}$ in~$E$. Indeed,
\begin{align*}
\bigl|T(x_k,k/n)&-T(x_{k+1},(k+1)/n)\bigr|\le a_{i_1j_1}a_{i_2j_2}\cdots a_{i_\ell j_\ell}|x_k-x_{k+1}|+b_{i_1}b_{i_2}\cdots b_{i_\ell}/n\\
&\le b_{i_1}b_{i_2}\cdots b_{i_\ell}(\epsilon_0|x_k-x_{k+1}|/2+1/n)
\le b_{i_1}b_{i_2}\cdots b_{i_\ell}(\epsilon_0/2+\epsilon_0/2)\\
&=\epsilon_0 b_{i_1}b_{i_2}\cdots b_{i_\ell}\le \epsilon_0 \bigl|T(x_0,0)-T(x_n,1)\bigr|=\epsilon_0|\xi-\xi'|.
\end{align*}
Thus, the claim is true and the proof of necessary part is completed.

\subsection{Proof of sufficiency}
Suppose that $E$ is totally disconnected and there exists an $i\in\{1,2,\dots,m\}$ such that $n_i=0$, we will show that $E$ is uniformly disconnected. The proof is heavily dependent on the special geometrical structure of Lalley-Gatzouras sets. Roughly speaking, the condition $n_i=0$ for some~$i$ means that the set~$E$ is well separated in the vertical direction. Thus, the focus is on the analysis of separation in the horizontal direction. 

We divide the proof into three steps. Step~1 studies the projection of~$E$ on the $y$-axis and the fiber of~$E$ on each horizontal line. Step~2 deals with the separation structure of the projection and the fibers. Step~3 gives the proof of sufficiency.

\subsubsection*{Step~1} In this step, we determine the projection of~$E$ on the $y$-axis and the fiber of~$E$ on each horizontal line.

Recall that the IFS of~$E$ is defined by~\eqref{eq:Sk}. Let
\[\mathcal I=\bigl\{i\in\{1,2,\dots,m\}\colon n_i\ne0\bigr\}.\]
For $i\in\mathcal I$, set $\phi_i(y)=b_iy+d_i$. The projection of $E$ to $y$-axis, denoted by $F$, is the self-similar set generated by the IFS $\Phi=\{\phi_i\colon i\in\mathcal I\}$, i.e.,
\begin{equation}\label{eq:proj}
	F=\bigcup_{i\in\mathcal I}\phi_iF.
\end{equation}

Now we study the fiber $E_y$ of~$E$ for $y\in F$, where $E_y=\{x\colon(x,y)\in E\}$. Clearly, $(x,y)\in E$ if and only if there are infinite words $i_1i_2\dotsc\in\mathcal I{^\mathbb N}$ and $j_1j_2\dotsc$ with $j_k\in\{1,2,\dots,n_{i_k}\}$ such that
\[(x,y)\in\bigcap_{k\ge1}S_{i_1j_1}\circ S_{i_2j_2}\dots\circ S_{i_kj_k}\bigl([0,1]^2\bigr).\]
Since $S_{ij}\colon(x,y)\mapsto(a_{ij}x+c_{ij},b_iy+d_i)$, this is equivalent to that 
\begin{equation}\label{eq:fiby}
	y\in\bigcap_{k\ge1}\phi_{i_1}\phi_{i_2}\circ\dots\circ\phi_{i_k} \bigl([0,1]\bigr)
\end{equation}
and that
\begin{equation}\label{eq:fibx}
	x\in\bigcap_{k\ge1}\psi_{i_1j_1}\psi_{i_2j_2}\circ\dots\circ\psi_{i_kj_k} \bigl([0,1]\bigr),\quad\text{where $\psi_{ij}(x)=a_{ij}x+c_{ij}$}.
\end{equation}

For $i\in\mathcal I$, let
\begin{equation}\label{eq:Psi}
  \Psi_i=\bigl\{\psi_{ij}\colon j\in\{1,2,\dots,n_i\}\bigr\}.
\end{equation}
For a finite word $i_1i_2\ldots i_k\in\mathcal I^k$, let
\[K_{i_1\dots i_k}=\bigcup_{f_1\in\Psi_{i_1},\dots,f_k\in\Psi_{i_k}}
f_1\circ\dots\circ f_k[0,1].\]
For an infinite word $\mathbf i=i_1i_2\ldots\in\mathcal I{^\mathbb N}$, let
\begin{equation}\label{eq:fiber}
	K_{\mathbf i}=\bigcap_{k\ge1} K_{i_1i_2\dots i_k}.
\end{equation}
Then \eqref{eq:fibx} is equivalent to $x\in K_{\mathbf i}$. Given $y\in F$, an infinite word $\mathbf i=i_1i_2\ldots\in\mathcal I^{\mathbb N}$ is called the coding of~$y$ if \eqref{eq:fiby} holds. Note that each point $y\in F$ belongs to at most two different set $\phi_i(F)$ and $\phi_j(F)$, where $i,j\in\mathcal I$. Hence every $y\in F$ has at most two different codings. 
Consequently, we have
\begin{lem}\label{l:fibery}
	For $y\in F$, the fiber~$E_y$ has the form
	\begin{equation}\label{eq:fibery}
		E_y=
		\begin{cases}
		K_{\mathbf i},&\text{if $y$ has only one coding $\mathbf i$};\\
		K_{\mathbf i}\cup K_{\mathbf j},&
		\text{if $y$ has two codings $\mathbf i$ and $\mathbf j$}.
		\end{cases}
	\end{equation}
\end{lem}

Although the fibers $E_y$'s are different for different $y$'s, the set $K_{\mathbf i}$ is close to~$K_{\mathbf j}$ in the Hausdorff metric if the coding~$\mathbf i$ is close to~$\mathbf j$. Recall that, for any two nonempty compact sets~$A$ and~$B$, the Hausdorff metric~$\hd$ is defined by
\[\hd(A,B)=\max\Bigl(\max_{x\in A}\dist(x,B),\max_{x\in B}\dist(x,A)\Bigr).\]
In the following, we use $|A|$ to denote the diameter of the set~$A$.
\begin{lem}\label{l:hd}
	For $\mathbf i=i_1i_2\ldots i_ki_{k+1}\ldots$ and $\mathbf j=i_1i_2\ldots i_{k}j_{k+1}\ldots$, where $i_{k+1}\ne j_{k+1}$, we have
	\begin{equation}\label{eq:hd}
		\hd(K_{\mathbf i},K_{\mathbf j})\le
		\max_{f_1\in\Psi_{i_1},\dots,f_k\in\Psi_{i_k}}
		\bigl|f_1\circ\dots\circ f_k[0,1]\bigr|=
		a^*_{i_1}a^*_{i_2}\cdots a^*_{i_k},
	\end{equation}
	where $\Psi_i$'s are defined by~\eqref{eq:Psi} and $a^*_i=\max_ja_{ij}$ for $i\in\mathcal I$.
\end{lem}
\begin{proof}
	By symmetry, it suffices to show that
	\[\dist(x,K_{\mathbf j})\le a^*_{i_1}a^*_{i_2}\cdots a^*_{i_k}
	\quad\text{for $x\in K_{\mathbf i}$}.\]
	Fix an $x\in K_{\mathbf i}$. By the definition of $K_{\mathbf i}$, there are $f_1\in\Psi_{i_1},\dots,f_k\in\Psi_{i_k}$ such that $x\in f_1\circ\dots\circ f_k[0,1]$. Since $\mathbf j=i_1i_2\ldots i_{k}j_{k+1}\ldots$, there is an $x'\in K_{\mathbf j}\cap f_1\circ\dots\circ f_k[0,1]$. Hence
	\[\dist(x,K_{\mathbf j})\le|x-x'|\le\bigl|f_1\circ\dots\circ f_k[0,1]\bigr|\le a^*_{i_1}a^*_{i_2}\cdots a^*_{i_k}.\qedhere\]
\end{proof}

\subsubsection*{Step~2}

This step studies the separation structure of the projection~$F$ and the sets~$K_{\mathbf i}$'s. We begin with the projection~$F$. 

To avoid confusion with the Hausdorff metric, we use $\dist$ to denote the Euclidean distance and $\dist(A_1,A_2):=\inf\{|x-y|\colon x\in A_1,y\in A_2\}$ for nonempty sets~$A_1,A_2$. For $\delta>0$, let
\[\mathcal I_\delta=\{i_1i_2\dots i_k\in\mathcal I^k\colon k\ge1,
b_{i_1}b_{i_2}\cdots b_{i_k}\le\delta<b_{i_1}b_{i_2}\cdots b_{i_{k-1}}\}.\]
We say $w,w'\in\mathcal I_\delta$ are $\delta$-connected if there exist
$w_0(=w),w_1,\dots,w_n(=w')\in\mathcal I_\delta$ such that
$\dist(\phi_{w_k}(F),\phi_{w_{k+1}}(F))\le\delta$ for
$k=0,1,\dots,n-1$. Here $\phi_w=\phi_{i_1}\circ\phi_{i_2}\circ\dots\circ\phi_{i_\ell}$ for $w=i_1i_2\dots i_\ell$, where $\phi_i$'s are defined in~\eqref{eq:fiby}. For $w\in\mathcal I_\delta$, let
\begin{equation}\label{eq:Idw}
	\mathcal I_\delta(w)=\{v\in\mathcal I_\delta\colon
	\text{$v$ and $w$ are $\delta$-connected}\}.
\end{equation}

The following lemma describes the separation structure of~$F$.
\begin{lem}\label{l:Fud}
	Suppose that there exists $i\in\{1,2,\dots,m\}$ with $n_i=0$, then
	\begin{equation}\label{eq:dcn}
		L:=\sup_{\delta,w}\#\mathcal I_\delta(w)<\infty.
	\end{equation}
\end{lem}
\begin{proof}
	Since $n_i=0$ for some $i\in\{1,2,\dots,m\}$, there exists an interval~$I$ such that $I\subset[0,1]\setminus F$. Let $\eta=|I|$ be the length of~$I$. Write $b_*=\min_{1\le i\le m}b_i$.
	
	If $\delta\ge\eta b_*$, we have
	\[\#\mathcal I_\delta(w)\le\#\mathcal I_\delta\le\delta^{-1}
	\le\eta^{-1}b_*^{-1}\quad\text{for all $w\in\mathcal I_\delta$}.\]
	
	If $\delta<\eta b_*$, fix a $w\in\mathcal I_\delta$. Write $\delta'=\delta/(\eta b_*)$. We claim that there are at most two different $u\in\mathcal I_{\delta'}$ such that
	\begin{equation}\label{eq:phiu}
		\phi_u(F)\cap\bigcup_{v\in\mathcal I_\delta(w)}\phi_v(F)\ne\emptyset.
	\end{equation}
	It follows from the claim that
	\[\bigcup_{v\in\mathcal I_\delta(w)}\phi_v(F)\subset\phi_{u_1}(F)\cup
	\phi_{u_2}(F)\quad\text{for some $u_1,u_2\in\mathcal I_{\delta'}$}.\]
	(Here we allow that $u_1=u_2$. In this case, $u_1$ is the only one word in~$\mathcal I_{\delta'}$ satisfying~\eqref{eq:phiu}). Consequently,
	\[ \#\mathcal I_\delta(w)\le\frac{|\phi_{u_1}(F)|+|\phi_{u_2}(F)|}{\min_{v\in\mathcal I_\delta(w)}|\phi_v(F)|}\le\frac{2\delta'|F|}{\delta b_*|F|}=2\eta^{-1}b_*^{-2}. \]
	Combining this with the case $\delta\ge\eta b_*$, we have
	\[L=\sup_{\delta,w}\#\mathcal I_\delta(w)\le 2\eta^{-1}b_*^{-2}. \]
	
	It remains to verify the claim. Indeed, if otherwise, suppose that there are distinct $u_1,u_2,u_3\in\mathcal I_{\delta'}$ satisfying~\eqref{eq:phiu}. Without loss of generality, suppose that $\phi_{u_1}(F)$ is on the left side of~$\phi_{u_2}(F)$ and $\phi_{u_3}(F)$ is on the right side of~$\phi_{u_2}(F)$. Since $I\subset[0,1]\setminus F$, we have
	\[ \bigcup_{v\in\mathcal I_\delta(w)}\phi_v(F)=A\cup B, \]
	where $A$ contains all the points which are on the left side of~$\phi_{u_2}(I)$ and $B$ contains all the points which are on the right side of~$\phi_{u_2}(I)$. Clearly, $A$ and $B$ are nonempty and $\dist(A,B)\ge|\phi_{u_2}(I)|>\delta'\eta b_*=\delta$. This contradicts the definition of~$\mathcal I_\delta(w)$.
\end{proof}

Now we turn to the separation structure of $K_{\mathbf i}$. 
\begin{lem}\label{l:Ki}
	There exists a constant $\lambda>0$ such that, for each open interval
	$I$ and each infinite word $\mathbf i$, we can find an open interval
	$J\subset I$ satisfying $|J|\ge\lambda|I|$ and $J\cap K_{\mathbf
	i}=\emptyset$.
\end{lem}
\begin{proof}
Fix an open interval $I$ and an infinite word $\mathbf i=i_1i_2\ldots$. Divide $I$ into three intervals of equal length, let $I'$ be the middle open
interval. If $I'\cap K_{\mathbf i}=\emptyset$, we can take $J=I'$ and
$\lambda=1/3$. 

Otherwise, by~\eqref{eq:fiber}, we can find a set $A$ of the form
\[A=\psi_{i_1j_1}\circ\psi_{i_2j_2}\circ\dots\circ\psi_{i_kj_k}[0,1]\]
such that $A\cap I'\ne\emptyset$ and $A\subset I$. We further require that $A$ is the set of largest length satisfying the above conditions. We have
\begin{equation}\label{eq:|A|}
  |A|\ge3^{-1}a_*|I|,\quad\text{where $a_*=\min_{i,j}a_{ij}$}.
\end{equation}
To see this, suppose otherwise $|A|<3^{-1}a_*|I|$ and let
\[A^*=\psi_{i_1j_1}\circ\psi_{i_2j_2}\circ\dots\circ\psi_{i_{k-1}j_{k-1}}[0,1].\]
Then $|A^*|<a_{i_kj_k}|A|<3^{-1}I$ and $A^*\cap I'\supset A\cap I'\ne\emptyset$. Since $I'\subset I$ is the middle open interval of length~$|I|/3$, we have $A^*\subset I$. Therefore $A^*$ also satisfies the conditions and $|A^*|>|A|$. This contradicts the fact that $|A|$ is largest.

Since $E$ is totally disconnected, for each $i\in\mathcal I$, we can find an open interval
$I_i\subset[0,1]\setminus\bigcup_{j=1}^{n_i}\psi_{ij}[0,1]$. Let
\[J=\psi_{i_1j_1}\circ\psi_{i_2j_2}\circ\dots\circ\psi_{i_kj_k} I_{i_{k+1}},\]
then $J\subset A\subset I$, $J\cap K_{\mathbf i}=\emptyset$ and
$|J|=|A|\cdot|I_{i_{k+1}}|\ge3^{-1}a_*|I_{i_{k+1}}|\cdot|I|$ by~\eqref{eq:|A|}. Thus, $J$ is desired and we can take
\[\lambda=3^{-1}a_*\cdot\min_{i\in\mathcal I}|I_i|.\qedhere\]
\end{proof}

\subsubsection*{Step~3}
We can prove the sufficiency of Theorem~\ref{t:UD} now. Let $(x_0,y_0)\in E$ and $\delta>0$ sufficiently small. Pick $w\in\mathcal I_\delta$ such that $y_0\in\phi_wF$. Let
\begin{equation}\label{eq:B}
B=\bigcup_{v\in\mathcal I_\delta(w)}\phi_vF.
\end{equation}
By the definition of $\mathcal I_\delta(w)$ (see~\eqref{eq:Idw}), we have
\begin{equation}\label{eq:ydc}
(y_0-\delta,y_0+\delta)\cap F\subset B\quad\text{and}\quad
\dist(B,F\setminus B)>\delta.
\end{equation}
By~\eqref{eq:dcn},
\begin{equation}\label{eq:diameter}
|B|\le 2L\delta.
\end{equation}

For each $v\in\mathcal I_\delta(w)$, pick an infinite word $\mathbf i_v$ which starts with the finite word~$v$. Let
\begin{equation}\label{eq:G}
G=\bigcup_{v\in\mathcal I_\delta(w)}K_{\mathbf i_v}.
\end{equation}
Let
\[h_\delta=\max_{\substack{i_1i_2\ldots i_k\in\mathcal I_\delta\\
i_\ell j_\ell\in\mathcal D\ \text{for $1\le\ell\le k$}}}
\frac{a_{i_1j_2}a_{i_2j_2}\cdots a_{i_kj_k}}{b_{i_1}b_{i_2}\cdots
b_{i_k}},\] 
note that $h_\delta\to0$ as $\delta\to0$. By Lemma~\ref{l:fibery} and~\ref{l:hd}, for all $y\in B$, we have
\begin{equation*}
E_y\subset G_{\delta h_\delta}:=\{x\colon\dist(x,G)\le \delta h_\delta\}.
\end{equation*}
It follows that
\begin{equation}\label{eq:GG}
\{(x,y)\in E\colon y\in B\}\subset G_{\delta h_\delta}\times B.
\end{equation}

Combining~\eqref{eq:G} and $\#\mathcal I_\delta(w)\le L$, using Lemma~\ref{l:Ki} at most $L$ times, we have, for each open interval $I$, there exists an open interval $J\subset I$ satisfying $J\cap G=\emptyset$ and
$|J|\ge\lambda^L|I|$. Let $I'=(x_0-2\delta,x_0-\delta)$ and
$I''=(x_0+\delta,x_0+2\delta)$, suppose $J'=(a',b')\subset I'$ and
$J''=(a'',b'')\subset I''$ are such subintervals. Let $G'=G\cap[b',a'']$
and $G''=G\setminus G'$, then $\dist(G',G'')\ge\lambda^L\delta$.

Finally, let $\delta$ be so small that $\delta
h_\delta<3^{-1}\lambda^L\delta$, then
\[G_{\delta h_\delta}=G'_{\delta h_\delta}\cup G''_{\delta h_\delta}\quad
\text{and}\quad\dist(G'_{\delta h_\delta},G''_{\delta h_\delta})
\ge3^{-1}\lambda^L\delta,\]
where $G'_{\delta h_\delta}=\{x\colon\dist(x,G')\le \delta h_\delta\}$ and $G''_{\delta h_\delta}=\{x\colon\dist(x,G'')\le \delta h_\delta\}$. We also have
\[(x_0-\delta,x_0+\delta)\subset G'_{\delta h_\delta}\quad\text{and}\quad
|G'_{\delta h_\delta}|\le4\delta.\]
Combining this with~\eqref{eq:B}, \eqref{eq:ydc}, \eqref{eq:diameter}
and~\eqref{eq:GG}, we have
\[(x_0-\delta,x_0+\delta)\times(y_0-\delta,y_0+\delta)\cap E
\subset(G'_{\delta h_\delta}\times B)\cap E,\quad
|E\cap(G'_{\delta h_\delta}\times B)|\le8L\delta\]
and
\[\dist\bigl(E\cap(G'_{\delta h_\delta}\times B),
E\setminus(G'_{\delta h_\delta}\times B)\bigr)\ge3^{-1}\lambda^L\delta.\]
This means that the set~$E$ is uniformly disconnected and the proof of sufficient part is completed.

\section{Gap sequence of Lalley-Gatzouras sets}\label{sec:GS}

This section is devoted to the proof of Theorem~\ref{t:GS}. We divide the proof into three lemmas.

\begin{lem}[\cite{DeWaX15,MiXiX17}]\label{l:gs}
	Let $\{\alpha_k\}_{k\ge1}$ be the gap sequence of a compact set $A$. Then
	\[
	\alpha_k\asymp k^{-1/\gamma}\iff \Num(\delta, A)\asymp \delta^{-\gamma}.
	\]
	Here $\gamma>0$.
\end{lem}

\begin{lem}\label{l:nmb}
	If $A\subset\mathbb R^d$ is a uniformly disconnected compact set, then
	\[ \Num(\delta, A)\asymp N_\delta(A)\quad
	\text{for $\delta>0$}. \] 
	Here $N_\delta(A)$ denotes the smallest number of sets of diameter~$\delta$ that cover~$A$. 
\end{lem}
\begin{proof}
	Recall that $\Num(\delta,A)$ is the number of $\delta$-equivalent classes of~$A$. Fix $\delta>0$, for $x\in A$, let $A_x$ be the set of all points $\delta$-equivalent to~$x$.
	
	By the definition of uniform disconnectedness, we have $|A_x|\le c\delta$ for all $x\in A$ and a constant $c$. Therefore,
	\begin{equation}\label{eq:Nc}
		\Num(\delta,A)\ge N_{c\delta}(A).
	\end{equation}
	
	By the definition of $\delta$-equivalence, we have $\dist(A_x,A_y)>\delta$ if $A_x\ne A_y$. Therefore,
	\begin{equation}\label{eq:N2}
		\Num(\delta,A)\le N_{\delta/2}(A).
	\end{equation}
	
	Now notice that in Euclidean spaces, we always have $N_{c\delta}(A)\asymp N_\delta(A)\asymp N_{\delta/2}(A)$. Together with~\eqref{eq:Nc} and~\eqref{eq:N2}, this completes the proof.
\end{proof}

\begin{lem}\label{l:nmbx}
	Let $E$ be a Lalley-Gatzouras set, then 
	\[ N_\delta(E)\asymp\delta^{-\bdim E}. \]
\end{lem}
\begin{proof}
	This lemma is proved by Lalley and Gatzouras~\cite{LalGa92}, although they didn't give an explicit statement regarding this. See the process of proofs of Lemma~2.1, Lemma~2.4 and Theorem~2.4 in~\cite{LalGa92}.
\end{proof}

Clearly, Theorem~\ref{t:GS} follows from Theorem~\ref{t:UD} and the three lemmas above.

\subsection*{An unsolved problem}

We end this paper with the following unsolved problem:
\begin{quote}\itshape 
	What is the limit behavior of gap sequence of Lalley-Gatzouras set when $n_i\ne0$ for all~$i$?
\end{quote}

The authors~\cite{MiXiX17} solved this problem for McMullen sets, but the argument is not applicable for Lalley-Gatzouras sets.

\end{document}